\documentclass[12pt]{article}
\usepackage{amssymb,amsmath,amsfonts}
\usepackage{amsthm}
\usepackage{color}
\numberwithin{equation}{section} \theoremstyle{plain}
\newtheorem{theorem}{Theorem}[section]
\newtheorem{corollary}[theorem]{Corollary}
\newtheorem{lemma}[theorem]{Lemma}
\theoremstyle{definition}
\newtheorem{definition}{Definition}[section]

\theoremstyle{remark}

\begin{document}

\title{On Lin's condition for products of random variables }\author{Alexander Il'inskii$^1$, Sofiya  Ostrovska$^2$}
\date{}
\maketitle

\begin{center}
\textit{ $^1$Department of Fundamental Mathematics,  Karazin National University, Kharkov, Ukraine}
\end{center}
\begin{center}
\textit{ $^2$Department of
Mathematics,  Atilim University,   Ankara, Turkey}
\end{center}

\begin{abstract} The paper presents an elaboration of some results on Lin's conditions. A new proof of the fact that  if densities of independent random variables $\xi_1$ and $\xi_2$ satisfy Lin's condition, the same is true for their product is presented. Also, it is shown that without the condition of independence, the statement is no longer valid.
\end{abstract}

\noindent {\small \textbf{Keywords:} random variable, absolutely continuous  distribution, Lin's condition
\vspace{0.2cm}}

\noindent {\small \textbf{Mathematics Subject Classifications:}
60E05}\vspace{0.2cm}

\section{Introduction}

Lin's condition plays a significant role in establishing `checkable' conditions for the moment (in)determinacy of probability distributions. This condition expresses certain regularity in the behaviour of probability densities. Given a probability density $f$, the tool used here is a function $L_f$ which was brought into consideration by G. D. Lin \cite{lin} and called Lin's function in subsequent researches starting from \cite{bernoulli}. The function is defined as follows.

\begin{definition} Let $f$ be a probability density continuously differentiable on $(0,\infty)$. The function
\begin{equation} \label{linfun}
L_f(x):=-\frac{xf^\prime(x)}{f(x)}
\end{equation}
is called \textit{Lin's function} of $f$.
\end{definition}

Clearly, Lin's function of $f$ is defined only at the points where $f$ does not vanish. In this work, we deal only with probability densities of positive random variables whose Lin's functions are defined for all $x>0.$ In particular, it is assumed that all densities do not vanish for all $x>0,$ that is Lin's function for them is well-defined and, in addition, only continuously-differentiable densities are considered. For such densities, the following condition was first considered by G. D. Lin in \cite{lin} with regard to the problem of moments.

\begin{definition}\label{lincond} Let $f\in C^1(0,\infty)$ be a probability density of a positive random variable. It is said that $f$ satisfies \textit{ Lin's condition} on $(x_0,\infty)$ if $L_f(x)$ is monotone increasing on $(x_0,\infty)$ and $\displaystyle \lim_{x\rightarrow +\infty}L_f(x)=+\infty$.
\end{definition}

 As this condition is used widely to investigate the moment determinacy of absolutely continuous probability distributions (see, \cite{recent, stirzaker} and references therein), it is a natural question to ask which operations on random variables preserve Lin's condition. Recently, Kopanov and Stoyanov in \cite{kopanov} established that if a density $f\in C^1(0,+\infty)$ of a random variable $X$ satisfies Lin's condition, then the densities of $X^r, r>0$ and $\ln X$ also satisfy Lin's condition. Further, if $L_f(x)/x\rightarrow +\infty$ as $x\rightarrow +\infty,$ then the density of $e^X$ also satisfies Lin's condition. In the same article \cite{kopanov}, it has been stated that if $X_1$ and $X_2$ are \textit{independent} positive random variables whose densities satisfy Lin's condition, then the density of their product also satisfies Lin's condition.  The approach suggested in \cite{kopanov} is based on the application of the mean value theorem for integrals. In this work, a different approach is proposed, which may be used in other problems, such as estimation of moments. Furthermore, it is demonstrated that the condition of independence is crucial here. In general, the statement is not true for the product of  dependent random variables whose densities satisfy Lin's condition.

\section{Statement of Results}

The first result of this work has been presented in \cite{kopanov}, and its proof based on the application of the mean value theorem is given in \cite{kopanov2}.  In the present paper, an alternative proof is provided which uses a different  technique. 

\begin{theorem}\label{th1} $\cite{kopanov}$ If $\xi_1$ and $\xi_2$ are positive independent random variables whose densities $f_1$ and $f_2$ satisfy Lin's condition on $(0,+\infty)$, then the density $g$ of their product satisfies Lin's condition on $(0,+\infty)$.
\end{theorem}

Evidently, the result can be extended by induction on the product of $n$ independent random variables.

The next theorem demonstrates that the condition of $\xi_1$ and $\xi_2$ being independent is crucial for the validity of the statement and, in general, it cannot be left out whatever the densities of $\xi_1$ and $\xi_2$ are.

\begin{theorem}\label{th2} Let  $f_1$ and $f_2$  be two densities of the positive random variables satisfying Lin's condition on $(0,+\infty)$.
Then there exists a random vector $(\xi_1,\xi_2)$ with absolutely continuous distribution such that the coordinates $\xi_1$ and $\xi_2$ have densities $f_1$ and $f_2$ respectively, the density $g$ of the product $\xi_1\cdot\xi_2$ is continuously differentiable on $(0,+\infty)$ and  the following relations are valid\,$:$
\begin{equation}\label{limsupinf}
\limsup_{x\rightarrow +\infty} L_g(x)=+\infty\,,\quad\liminf_{x\rightarrow +\infty} L_g(x)=-\infty\,.
\end{equation}
\end{theorem}

Obviously, equalities \eqref{limsupinf} imply that $g$ does not satisfy Lin's condition on any interval $(x_0,\infty)$.

\section{Some auxiliary results}

To begin with, let us recall that, if $f(x,y)$ is a joint probability density of positive random variables $\xi_2$ and $\xi_2$,  then the density $g$ of their product is given by:
\begin{equation}\label{product}
g(x)=\int_0^{+\infty}f\left(t, \frac{x}{t}\right)\frac{dt}{t}\,.
\end{equation}
See, for example \cite[page 618, formula (18.5-17)]{korn}. With the help of \eqref{product}, the next useful outcome can be derived.
\begin{lemma}\label{lem1} $\cite{kopanov2}$ Let $\xi_1$ and $\xi_2$ be independent random variables whose densities $f_1$ and $f_2$ possess Lin's functions. If $g$ is the (continuous) density of the product $\xi_1\cdot\xi_2$, then\,$:$
\begin{equation}\label{lg}
L_g(x)=\frac{1}{g(x)}\int_0^\infty f_1\left(\frac{x}{t}\right)f_2(t)L_{f_2}(t)\frac{dt}{t}
=\frac{1}{g(x)}\int_0^\infty f_1\left(\frac{x}{t}\right)f_2(t)L_{f_1}\left(\frac{x}{t}\right)
\frac{dt}{t}\,.
\end{equation}
\end{lemma}

\begin{lemma}\label{lem2} Let $f$ be a probability density such that $L_f(x)$ is monotone increasing for all $x>0.$ Then, for every $0<a<b,$ the function \begin{equation} \label{tau}
\tau(x):=\frac{f(ax)}{f(bx)}
\end{equation}
is monotone increasing in $x$.
\end{lemma}

\begin{proof}
Indeed,
\begin{equation*}
\tau^\prime(x)=\frac{f(ax)}{xf(bx)}\left[L_f(bx)-L_f(ax)\right]>0\;\;\forall x>0\,.
\end{equation*}
\end{proof}

\section{Proofs of the Theorems}

\textbf{Proof of Theorem \ref{th1}}. \textbf{1}. First, we are going to prove that $L_g(x)$ is monotone increasing on $(0,+\infty)$.  Select $0<x<y$ and consider $L_g(y)-L_g(x)$. By virtue of \eqref{product} and \eqref{lg}, one has:
\begin{equation*}
\begin{split}
&L_g(y)-L_g(x)=\\
&\frac{\int_0^{+\infty}f_1(y/v)f_2(v)L_{f_2}(v)\frac{dv}{v}}{\int_0^{+\infty}f_1(y/v)f_2(v)\frac{dv}{v}}
-\frac{\int_0^{+\infty}f_1(x/u)f_2(u)L_{f_2}(u)\frac{du}{u}}{\int_0^{+\infty}f_1(x/u)f_2(u)\frac{du}{u}}=\\
&\frac{1}{g(x)g(y)}\int_0^{+\infty}\int_0^{+\infty} f_1(y/v)f_2(v)f_1(x/u)f_2(u)\left[L_{f_2}(v)-L_{f_2}(u) \right]\frac{dudv}{uv}=\\
&\frac{1}{g(x)g(y)}\left[\int\!\int_{A_1}+ \int\!\int_{A_2}\right]\,,
\end{split}
\end{equation*}
where $\displaystyle A_1=\{(u,v):u>v\}$ and $\displaystyle A_2=\{(u,v):u<v\}$. Now, interchanging $u$ and $v$ in $\int\!\int_{A_2}$, one derives:
\begin{equation*}
\int\!\int_{A_2}=-\int\!\int_{A_1}f_1(y/u)f_1(x/v)f_2(u)f_2(v)\left[L_{f_2}(v)-L_{f_2}(u)\right]\frac{dudv}{uv}\,.
\end{equation*}
Therefore,
\begin{equation}\label{lgx}
\begin{split}
&L_g(y)-L_g(x)=\\
&\int\!\int_{A_1}\frac{f_2(u)f_2(v)}{uv}\left[L_{f_2}(v)-L_{f_2}(u)\right]\cdot\left[f_1(y/v)f_1(x/u)-f_1(y/u)f_1(x/v)
\right]dudv\,.
\end{split}
\end{equation}
Now, consider the expressions in the both brackets. Since $u>v$ in $A_1$ and $L_{f_2}$ is strictly increasing, it follows that
the first one is negative everywhere in $A_1$. The second one can be rewritten as follows:
\begin{align*}
f_1(y/v)f_1(x/u)-f_1(y/u)f_1(x/v)=f_1(y/v)f_1(x/u)\left[1-\frac{\tau(y)}{\tau(x)}\right]\,,
\end{align*}
where $\tau (x)$ is defined by \eqref{tau} with $a=1/u<1/v=b.$ Lemma \ref{lem2} implies that $\tau(y)>\tau(x)$ and, as a result $f_1(y/v)f_1(x/u)-f_1(y/u)f_1(x/v)<0$ in $A_1$. To summarize, the integrand in \eqref{lgx} is positive, whence  $L_g
(y)>L_g(x)$ whenever $y>x$, as stated.

\bigskip
\textbf{2}.
At this stage, we are going to prove that $L_g(x) \rightarrow +\infty$ as $x\rightarrow +\infty$.
By virtue of Lemma \ref{lem1}, formula \eqref{lg}, one has:
\begin{align*}
g(x) L_g(x)= \int_0^\infty f_1\left(\frac{x}{t}\right)f_2(t)L_{f_2}(t)\frac{dt}{t}=
\int_0^\infty f_1\left(\frac{x}{t}\right)f_2(t)L_{f_1}\left(\frac{x}{t}\right)\frac{dt}{t}
\end{align*}
which, after substitution $t\mapsto \sqrt{x}t,$ leads to:
\begin{equation*}
2g(x)L_g(x)=\int_0^\infty f_1\left(\frac{\sqrt{x}}{t}\right)f_2(\sqrt{x} t)\left[
L_{f_1}\left(\frac{\sqrt{x}}{t}\right)+L_{f_2}(\sqrt{x}t)\right]\frac{dt}{t}\,.
\end{equation*}
Now, since both $L_{f_1}$ and $L_{f_2}$ are increasing in their arguments, it follows that, for every $t>0$,
\begin{align*}
L_{f_1}\left(\frac{\sqrt{x}}{t}\right)+L_{f_2}(\sqrt{x}t)\geq \min\{L_{f_1}(\sqrt{x}),L_{f_2}(\sqrt{x})\}=:\tilde{L}(\sqrt{x})\,.
\end{align*}
Therefore,
\begin{align*}
2g(x)L_g(x)\geq \tilde{L}(\sqrt{x})\int_0^\infty f_1\left(\frac{\sqrt{x}}{t}\right)f_2(\sqrt{x} t)
\frac{dt}{t}=\tilde{L}(\sqrt{x})g(x)\,,
\end{align*}
implying
$$
L_g(x)\geq \frac{1}{2} \tilde{L}(\sqrt{x})\,.
$$
The statement now follows.

$\hfill\Box$

\begin{corollary} If $f_1=f_2$, then
$$
L_g(x)\geq \frac{1}{2} L_f(\sqrt{x})\,.
$$
\end{corollary}

\textbf{Proof of Theorem \ref{th2}}. Let us denote
$C(a,b;r):=\{(x,y)\in\mathbb{R}^2:(x-a)^2+(y-b)^2=r^2\}$,
$D(a,b;r):=\{(x,y)\in\mathbb{R}^2:(x-a)^2+(y-b)^2\leq r^2\}$
the circumference and the disc with the center $(a,b)$ and radius $r$ respectively.
For any $0<a<v$, consider the square:
$$
K=\{(x,y):v-a\leq x,y\leq v+a\}\subset \mathbb{R}^2\,.
$$
Fix $0<r<a/4$ and
consider a function $\rho (x,y)\in C^\infty (\mathbb{R}^2)$ such that:

$(i)\;\;\rho(x,y)=1$ when $(x,y)\in D(0,0;r/2)$;

$(ii)\;\;\rho(x,y)=0$ when $(x,y)\notin D(0,0;r)$;

$(iii)\;\;0\leq \rho (x,y)\leq 1$ for all $(x,y)\in \mathbb{R}^2$.

Such a function can be constructed, for a example, in the following way.
Starting with $q(t)\in C^\infty[0,\infty)$ satisfying the conditions $q(t)=1$
for $t\in [0,r^2/4]$, $q(t)=0$ for $t>r^2$, and $q(t)$ is monotone decreasing on $(r^2/4, r^2)$, we set:
\begin{equation}\label{rhoo}
\rho(x,y):=q(x^2+y^2)
\end{equation}
which is a desired function.
Now, using function $\rho$, put:
\begin{equation}\label{nu}
\varphi(x,y)=\beta \cdot\sin (\nu xy)\cdot \rho(x-v-a/2,y-v-a/2)\,,
\end{equation}
where $\beta$ is a fixed number such that
$$
0<\beta <\min\{f_1(x)f_2(y):(x,y)\in K\}
$$
and  $\nu >0$ is a parameter whose value will be determined later. Obviously, $\varphi(x,y)\in
C^\infty\left(\mathbb{R}^2\right)$, $\varphi (x,y)\geq 0$,
\begin{equation*}
\varphi(x,y)=\begin{cases}\beta \cdot\sin (\nu xy),&\text{for $(x,y)\in D(v+a/2,v+a/2;r/2)$,}\\
0,&\text{for $(x,y)\notin D(v+a/2,v+a/2;r)$.}
\end{cases}
\end{equation*}
 Now, define
\begin{equation*}
f(x,y):=
f_1(x)f_2(y)-\varphi(x,y)+\varphi(x,y+a)-\varphi(x+a,y+a)+\varphi(x+a,y)\,.
\end{equation*}
Obviously, $ f(x,y)\geq 0$ for all $(x,y)\in \mathbb{R}^2$ and $ f(x,y)=0$ outside of the first quadrant. What is more, $f(x,y)$ is a joint probability density of some positive random variables, say, $\xi_1$ and $\xi_2$, whose marginal distributions have given densities $f_1$ and $f_2$ and, as such,
the densities of $\xi_1$ and $\xi_2$ satisfy Lin's condition on $(0,+\infty)$. What about the density $g$ of their product $\xi_1\cdot\xi_2$?

To derive the conclusion of this Theorem, notice that, for each
$$
z\in \left((v+a/2)^2-r^2/10, (v+a/2)^2+r^2/10\right)\,,
$$
hyperbola $\Gamma_z:=\{(x,y)\in \mathbb{R}_+\times \mathbb{R}_+: xy=z\}$
intersects both of the circumferences $S(v+a/2,v+a/2;r)$ and $S(v+a/2,v+a/2;r/2)$
at two distinct points. Denote the abscissas of these points by $x_1<x_2<x_3<x_4$. By formula \eqref{product}, the density
\begin{equation*}
\begin{split}
&g(z)=
\int_0^\infty f(x,\frac{z}{x})\frac{dx}{x} - \int_{x_1}^{x_2}\varphi\left(x,\frac{z}{x}\right)\frac{dx}{x}-\int_{x_2}^{x_3}
\varphi\left(x,\frac{z}{x}\right)\frac{dx}{x} - \int_{x_3}^{x_4}\varphi\left(x,\frac{z}{x}\right)\frac{dx}{x}\\
&=:p(z)-I_1(z)-I_2(z)-I_3(z)=:p(z)-I(z)\,.
\end{split}
\end{equation*}
Here, $p(z)$ is the (continuous) density of the product of independent random variables with densities $f_1$ and $f_2$.
By \eqref{nu}, one has
\begin{equation*}
I_2(z)=\int_{x_2}^{x_3}\beta\sin (\nu z)\frac{dx}{x}=\beta\sin(\nu z)\log\frac{x_3}{x_2}\,.
\end{equation*}
We notice that $\log(x_3/x_2)\geq c=c_{v,a,r}$.
As for $I_1(z)$ and $I_2(z),$ it can be observed that they have the same sign as $\sin(\nu z)$ whenever $\sin (\nu z)\neq 0$. Let $z_1, z_2, $ and $z_3$ be successive extreme points of $\sin(\nu z)$ falling into interval $\left((v+a/2)^2-r^2/10, (v+a/2)^2+r^2/10\right)$. To be specific, opt for
\begin{equation*}
\sin(\nu z_1)=1,\quad \sin (\nu z_2)=-1, \quad\mathrm{and}\quad\sin(\nu z_3)=1.
\end{equation*}
This can be achieved by taking large $\nu$, so that the extreme points become very close. Since for all these value of $z$, one has $\log(x_3/x_2)\geq c>0$, it follows that
$$
I(z_1)>I_2(z_1)\geq\beta\sin(\nu z_1)\cdot c=c\beta\,,
$$
$$
I(z_2)<I_2(z_2)\leq\beta\sin(\nu z_2)\cdot c=-c\beta\,,
$$
implying $I(z_1)-I(z_2)>2c\beta$.
Correspondingly, there exists $z_*\in (z_1,z_2)$ such that
$$
I'(z_{*})<-\frac{2c\beta}{\pi/\nu}=-\frac{2c\beta\nu}{\pi}\,,
$$
which can achieve arbitrarily large negative values for  $\nu$ sufficiently large.
Likewise, adding $z_3$, one obtains: $I(z_3)\geq c\beta$, whence
$$
I(z_3)-I(z_2)\geq 2c\beta$$ and, consequently,
$$
I'(z_{**})>\frac{2c\beta\nu}{\pi}\;\;\mathrm{for\;\;some}\;\;z_{**}\in (z_2,z_3)\,.
$$
Since $g^\prime(z)=p^\prime(z)-I^\prime(z)$ and $p^\prime(z)$
is bounded on $\left[(v-a)^2,(v+a)^2\right]$ by constant independent from $\nu,$ it follows that, for $\nu$ large enough, there exist points
$$
z_*, z_{**}\in \left((v+a/2)^2-r^2/10, (v+a/2)^2+r^2/10\right)
$$
such that $g^\prime(z_*)\geq A$ and $g^\prime (z_{**})\leq -B$ for any prescribed $A, B >0$.

Applying the same procedure to an infinite sequence of disjoint squares
\begin{equation*}
K_n=\{(x,y):v_n-a_n\leq x,y\leq v_n+a_n,\}, \;n\in \mathbb{N}
\end{equation*}
one derives the statement.

\section{Acknowledgement} The second author expresses her sincere gratitude to Prof. Jordan Stoyanov who provided her with a file of \cite{kopanov2} prior to its publication.

%\newpage

\end{document}